\documentclass[12pt]{amsart}
\usepackage{amsmath}
\usepackage{amssymb}
\usepackage{amsthm}
\usepackage{color,comment}
\usepackage[all]{xy}
\usepackage{graphicx}
\usepackage{amscd,color}
\usepackage[shortlabels]{enumitem}

\renewcommand{\Im}{\operatorname{Im}}
\renewcommand{\Re}{\operatorname{Re}}

\newcommand{\cala}{{\mathcal A}}
\newcommand{\calb}{{\mathcal B}}

\newcommand{\calf}{{\mathcal F}}

\newcommand{\calm}{{\mathcal M}}
\newcommand{\caln}{{\mathcal N}}

\newcommand{\calr}{{\mathcal R}}
\newcommand{\cals}{{\mathcal S}}

\newcommand{\C}{{\mathbb C}}
\newcommand{\D}{{\mathbb D}}

\newcommand{\M}{{\mathbb M}}
\newcommand{\N}{{\mathbb N}}

\newcommand{\R}{{\mathbb R}}

\newcommand{\Z}{{\mathbb Z}}
\newcommand{\chat}{{\widehat{\mathbb C}}}

\renewcommand{\hat}{\widehat}

\newcommand{\la}{\lambda}
\newcommand{\las}{\lambda^*}
\newcommand{\inv}{^{-1}}
\newcommand{\lo}{{\la_0}}

\newcommand{\ba}{{\bf a}}

\newtheorem{thm}{Theorem}[section]

\newtheorem{cor}[thm]{Corollary}
\newtheorem{prop}[thm]{Proposition}
\newtheorem{lemma}[thm]{Lemma}
\newtheorem{defn}{Definition}

\newtheorem{remark}{Remark}[section]

\def\M{\mathcal M}
\def\N{\mathcal N}

\def\F{\mathcal F}

\def\FT{\mathcal{FT}}

\def\Minf{\mathcal M_{\infty}}

\def\bfa{\mathbf a}

\def\CC{\mathbb C}

\def\ZZ{\mathbb Z}

\def\iota{\mathcal{INV}}

\newtheorem*{thmA}{Theorem A}
\newtheorem*{thmB}{Theorem B}
\newtheorem*{thmC}{Theorem C}

\newtheorem*{propA}{Proposition A}

\newtheorem*{corA}{Corollary A}
\newtheorem*{corB}{Corollary B}

\begin{document}
        \title{ Dynamics of Generalized Nevanlinna Functions}

\author{Tao Chen  and Linda Keen}

\address{}
\email{}

\thanks{}


\begin{abstract}
In the early 1980's, computers made it possible to observe that in complex dynamics, one often sees dynamical behavior reflected in parameter space and vice versa.
    This duality was first exploited by Douady, Hubbard and their students in early work on rational maps.   See \cite{DH,BH}  for example.

Here, we continue to study these ideas in the realm of transcendental functions.

 In \cite{KK1}, it was shown that for  the tangent family, $\lambda \tan z$,  the way the hyperbolic components meet at a point where the asymptotic value eventually lands on infinity reflects the dynamic behavior of the functions at infinity.  In the first part of this paper we show that this duality extends to a much more general class of transcendental meromorphic functions that we call {\em generalized Nevanlinna functions} with the additional property that infinity is not an asymptotic value.
 In particular, we show that in
   ``dynamically natural'' one dimensional slices of parameter space, there are ``hyperbolic-like'' components with a unique distinguished boundary point
 whose  dynamics reflect the behavior inside an asymptotic tract at infinity.  Our main result is  that {\em every} parameter point in such a slice for which  the asymptotic value eventually lands on a pole is such a distinguished boundary point.

In the second part of the paper, we apply this result to the families $\lambda \tan^p z^q$, $p,q \in \mathbb Z^+$, to prove that all hyperbolic components of period greater than $1$ are bounded.
  \end{abstract}

\maketitle
\section*{Introduction}

In the early 1980's, computers made it possible to observe that in complex dynamics, one often sees dynamical behavior reflected in parameter space and vice versa.
    This duality was first exploited by Douady, Hubbard and their students in early work on rational maps.   See \cite{DH,BH}  for example.
     Here, we continue to study these ideas in the realm of transcendental functions.

  In \cite{KK1}, it was shown that for  the tangent family, $\lambda \tan z$,  the way the hyperbolic components meet at a point where the asymptotic value eventually lands on infinity, reflects the dynamic behavior of the functions at infinity.  This is very different from what one sees for the exponential family $\exp z+c$ (see \cite{DFJ, RG, Sch}) where all the hyperbolic components are unbounded.  A crucial difference between these families is that for the exponential family infinity is an asymptotic value whereas for the tangent family it is not.

   In the first part of this paper we define  a much more general class of transcendental meromorphic functions that we call {\em generalized Nevanlinna functions}.   We show that  families of this class with the additional property that infinity is NOT an asymptotic value exhibit this duality; that is, in a one dimensional slice of the parameter space,  the way the hyperbolic components  meet at a point where the asymptotic value eventually lands on infinity reflects the dynamical behavior of the functions at infinity.

To describe these functions and state our theorems, we need some background. The dynamical plane of a meromorphic map is divided into two sets: the {\em Fatou or stable set} on which the iterates are well defined and form a normal family, and the {\em Julia set}, its complement.  The Julia set can be characterized as the closure of the set of repelling periodic points, or equivalently, the closure of the set of pre-poles, points that map to infinity after finite iteration.   A good introduction to meromorphic dynamics can be found in \cite{Ber}.

The points over which a meromorphic function is not a regular covering map are called singular values.  There are two types of singular values: {\em critical values}  (images of  zeroes of $f'$, the {\em critical points}) and {\em asymptotic values} (points $v = \lim_{t \to \infty} f(\gamma(t))$ where $\gamma(t) \to \infty$ as $t\to\infty$).  If an asymptotic value is isolated, the local inverse there is the logarithm.
We denote the class of meromorphic functions with finitely many singular values by $\FT$.  This class is particularly tractable:  families in $\FT$  have parameter spaces
with natural embeddings into $\CC^n$, where $n$ is a simple function of the number of singular values; all asymptotic values are isolated; their Fatou domains have a simple classification because  there are no wandering domains (see Section~\ref{basics} for definitions).

The focus in this paper is on the subclass $\Minf$ of functions in $\FT$ for which infinity is not an asymptotic value.  These
necessarily have infinitely many poles and  this  behavior at infinity has consequences for the dynamical and parameter spaces (see e.g. \cite{BK}).   In particular, it is very different from that for entire functions or transcendental functions with finitely many poles.

 A general principle in dynamics is that each stable dynamical phenomenon is ``controlled'' by a singular point.  For example, each attracting or parabolic periodic cycle always attracts a singular value.
   Using this principle, one can define  one dimensional slices of the parameter space of a family in $\Minf$ that are ``compatible with the dynamics" by keeping all but one of the dynamical phenomena fixed and letting the remaining one, controlled by the  ``free singular value $v$'', vary.    Looking for regions in the slice where the free singular value is attracted to an attracting cycle gives a picture of how the ``hyperbolic-like" components of the slice fit together around the bifurcation locus.

 {\em Dynamically natural} slices were defined in \cite{FK} for families of functions in  $\Minf$.
   It was shown there that for those slices, as for slices of parameter spaces of rational maps,  the bifurcation locus  contains parameters distinguished by functional relations.  There are {\em Misiurewicz points} where $v$ lands on a repelling periodic cycle and    {\em centers} where $v$ lands on a super-attracting cycle.    Another type of distinguished parameter,  not seen for rational maps,  is a {\em virtual cycle parameter}, where $v$ lands on a pole.

Dynamically natural slices contain two different kinds of ``hyperbolic-like'' domains in which the functions are all quasiconformally conjugate on their Julia sets: {\em capture components}, where the free asymptotic value is attracted to one of the fixed phenomena --- an attracting or super-attracting cycle, and {\em shell components},  where the free asymptotic value is the only singular value attracted to an attracting cycle. In a shell component, the period of the attracting cycle is constant and the multiplier map is a well-defined universal covering map onto the punctured disk.   Properties of shell components for general families in $\Minf$  were studied in detail in \cite{FK}.  In particular, it was proved that the boundary of every shell component contains a special point, the {\em virtual center} where the limit of the multiplier map is zero.   One of the main results proved there is

\begin{thmA} For families in $\Minf$, a virtual center on the boundary of a shell component in a dynamically natural slice is a virtual cycle parameter and any virtual cycle parameter on the boundary is a virtual center.
 \end{thmA}

In this paper, we concern ourselves with a special subclass in $\Minf$.  We start with the {\em Nevanlinna functions}, the family $\N_r \subset \FT$ of functions with $r>0$ asymptotic values and no critical values.  We next form the family $\tilde\calm_{p,q,r}$ by pre- and post-composing by polynomials of degrees $q$ and $p$ respectively.   We then look at $\calm_{p,q,r} \subset \Minf$, the subset of   functions in $\tilde\calm_{p,q,r}$ all of whose asymptotic values are finite.  We first define the embedding of $\calm_{p,q,r}$ into $\CC^{p+q+r+3}$ and then study shell components of dynamically natural slices for this embedding.  We include the proof, given originally in \cite{FK}, that

\begin{propA}
If $h$ is topologically conjugate to a meromorphic function $f=  P \circ g \circ Q$ in $\tilde\calm_{p,q,r}$, and if $h$ is meromorphic, then it is also in $\tilde\calm_{p,q,r}$;  that is there is a  function $\tilde{g} \in  \mathcal{N}_r $ and  polynomials $\tilde P, \tilde Q$ of degrees $p,q$ respectively such that $h = \tilde{P} \circ \tilde{g} \circ \tilde{Q}$.
\end{propA}

A corollary  is

\begin{corA}  There is a natural embedding of $\tilde\calm_{p,q,r}$ into $C^{p+q+r+4}$ and hence an embedding of $\calm_{p,q,r}$ into $C^{p+q+r+3}$.
\end{corA}

Theorem A does not preclude the existence of virtual cycle parameters that are not on the boundary of a shell component. For example, such a parameter might be  buried in the bifurcation locus.
Our first new theorem says this cannot happen.

\begin{thmB} In a dynamically natural slice in $\calm_{p,q,r}$, every virtual cycle parameter lies on the boundary of a shell component.
\end{thmB}

Corollary B, which follows directly from  Theorems  A  and B, says that in $\calm_{p,q,r}$ the notions of virtual center and virtual cycle parameter are equivalent.

\begin{corB}
 In a dynamically natural slice of  $\calm_{p,q,r}$,  every virtual cycle parameter is a virtual center and vice versa.
 \end{corB}

   In \cite{FK} it was proved that  for families in $\Minf$, shell components of period $1$ in dynamically natural slices  are always unbounded and it was conjectured that, in contrast to families of entire functions,  those of period greater than $1$ are always bounded.   The conjecture was proved true in the tangent family in \cite{KK1}.  In this paper we prove it for the generalization of the tangent family,  $\calf=\la\tan^p(z^q)$.  This is a  subfamily of $\calm_{p,q,r}$.

\begin{thmC}
Every shell component of period greater than $1$ in the $\lambda$ plane of the family $\calf$ is bounded.
\end{thmC}

The paper is organized as follows.   Part~\ref{part1} is a discussion of the general family $\calm_{p,q,r}$ and dynamically natural slices of its parameter space.  In Section~\ref{basics} we give a brief overview of the basic theory, set our notation and discuss the theorem of Nevanlinna, Theorem~\ref{thm:Nev}, which we use to define the class $\N_r$ of Nevanlinna functions.  Next, in Section~\ref{Mpqr}, we define the classes $\tilde\calm_{p,q,r}$ and $\calm_{p,q,r}$, and prove Proposition~A and Corollary~A.  In Section~\ref{dynamics} we  recall  the properties of Fatou components that we need and define dynamically natural slices of parameter spaces and the  shell components in these slices.    In Section~\ref{shell}, we define virtual cycle parameters and virtual centers, state Theorem~A and prove Theorem~B.   Part~\ref{part2} is a discussion of the special symmetric subfamily $\calf \subset \calm_{p,q,2}$.  In Section~\ref{tanpq} we classify the shell components by period and discuss the special properties of the components of periods $1$ and $2$.   Finally, in Section~\ref{thmC} we prove Theorem~C.

\part{The family $\M_{p,q,r}$} \label{part1}
\section{Basics and Tools}\label{basics}
 \subsection{Meromorphic functions.}
  In this paper,  unless we specifically say otherwise, we always assume that an entire or meromorphic map is transcendental and so has infinite degree.  If we mean a map of  finite degree we  call it polynomial or rational.  We also always assume infinity is an essential singularity.   We need the following definitions:

 \begin{defn}\label{sing pts} A point $v\in\C$ is called a {\em singular value}\footnote{In the literature these are sometimes called singularites of $f^{-1}$.} of $f$ if, for some small neighborhood of $v$, some branch of $f\inv$  is not well defined.   If $c$ is  a zero of $f'$, it is a critical point and its image $v=f(c)$ is a  {\em critical value}. The branch with $f^{-1}(v)=c$ is not well defined so a critical value is singular.   A point is also singular if there is a path $\gamma(t)$ such that $\lim_{t \to \infty} \gamma(t)=\infty$ and $\lim_{t \to \infty} f(\gamma(t))=v$.  The limit $v$ is singular and called an {\em asymptotic value}  of $f$.     Singular values may be critical, asymptotic or accumulations  of such points.   We denote the set of singular points  of $f$ by $\mathcal{S}_f$.
 \end{defn}

 If  an asymptotic value $v$ is isolated, and $\gamma$ is an asymptotic path for $v$, we can find a nested sequence of neighborhoods of $v$, $D_{r,v}$ with $r\to 0$,  and a particular branch $g$ of $f^{-1}$ such that $V_r=g(D_{r,v})$ is a nested sequence of neighborhoods containing $\gamma$ and $\cap g(D_{r,v}) = \emptyset$.   Then $r$ can be chosen small enough so that the  map $f: V_r \rightarrow D_{r,v} \setminus \{v\}$ is a universal covering map.  In this case $V_r$ is called an {\em asymptotic tract} for the asymptotic value $v$ and $v$ is called a {\em logarithmic singular value}.  The number of distinct asymptotic tracts of a given asymptotic value is called its {\em multiplicity}\footnote{Note this notion is not the same as multiplicity of a critical point.}.

 \begin{defn}\label{Julia} A ray $\beta$ approaching infinity is called a {\em Julia ray} or {\em Julia direction} for the meromorphic function $f$ if $f$ assumes all (but at most two) values infinitely often in any sector containing $\beta$.
 \end{defn}

An example to keep in mind is $e^z$ with asymptotic values at $0$ and $\infty$.  The asymptotic tracts are the left and right half planes and the Julia directions are parallel to the positive and negative imaginary axes.    Another example is $\tan z$ with asymptotic values $\pm i$, asymptotic tracts the upper and lower half planes and Julia directions parallel to the positive and negative real axes.

In this paper, whenever we talk about the number of critical points and/or asymptotic values, we tacitly assume that we count with multiplicity.

An entire function $f$ always has an asymptotic value at infinity.

We will only be interested in meromorphic functions with $\# \mathcal{S}_f < \infty$.  These are called functions of {\em finite type}.    Note that all the asymptotic values of a finite type function are isolated and so are logarithmic.

\subsubsection{  Nevanlinna Functions}
  Nevanlinna, in  \cite{Nev, Nev1}, Chap X1, characterized families of
meromorphic functions with finitely many asymptotic values, finitely many critical points  and a single essential singularity at infinity.    (See \cite{DK, KK1,EreGab} for further discussion.)

  Recall that the Schwarzian derivative of a function $g$ is defined by
 \begin{equation}\label{schw}
 S(g)=(g''/g')' - \frac{1}{2}(g''/g')^{2}.
 \end{equation}
Because  Schwarzian derivatives satisfy the cocycle relation
\[  S(f \circ g)(z) = S(f)(g'(z))^2 + S(g)(z) \]
and the Schwarzian derivative of a M\"obius transformation is zero,  solutions to equation~(\ref{schw}) are determined only up to
post-composition by a M\"obius transformation.

     Nevanlinna's theorem says
\begin{thm}[Nevanlinna]  \label{thm:Nev}
Every meromorphic function $g$ with   $p < \infty$ asymptotic values and $q < \infty$ critical points  has the property that its Schwarzian derivative is a rational function of degree $p +q -2$.
If $q=0$, the Schwarzian derivative is a polynomial $P(z)$.
In the opposite direction,  for every polynomial function $P(z)$ of degree $ p-2$,  the solution to
the Schwarzian differential equation $S(g)=P(z)$ is a meromorphic function with exactly $p$ asymptotic values and no critical points.
\end{thm}

\begin{remark}\label{truncated solutions} The proof of the first part of this theorem involves the construction of the function as a limit of  holomorphic functions whose Schwarzians are rational of bounded degree.   The proof of the second part of the  theorem involves understanding the asymptotic properties of solutions to the equation $S(g)=P(z)$.  In particular, there are exactly $p$ ``truncated solutions'' $g_0, \ldots, g_{p-1}$ that, for any $\epsilon >0$, have asymptotic developments of the form
\[ \log g_k(z) \sim (-1)^{k+1} z^{p/2} \]
defined in the sector $|\arg z - 2\pi k/p | < 3 \pi/p - \epsilon$.  Each $g_k$ is entire and  tends to zero as $z$ tends to infinity along each ray of the sector $|\arg z - 2\pi k/p | <  \pi/p$ and tends to infinity in the adjacent sectors.   The rays separating the sectors are the Julia rays for $g(z)$.   See \cite{Hille, Nev} for details.
\end{remark}

\begin{defn} We   denote the family of meromorphic functions with $p<\infty$ asymptotic values and no critical values by $\mathcal{N}_p$ and call the functions {\em Nevanlinna functions}. \end{defn}

One immediate corollary is that Nevanlinna functions cannot have exactly one asymptotic value.  Moreover, since polynomials of degree $p-2$ depend on their $p-1$ coefficients, and the solutions are determined up to post-compostion by M\"obius transformations which depend on three coefficents,
another  immediate corollary is
\begin{cor}\label{cor:dim Nev} Nevanlinna functions with $p<\infty$ asymptotic values  have a natural embedding into $\C^{p+2}$.
\end{cor}

We are interested in the dynamical systems generated by these Nevanlinna functions and some generalizations of them.  Infinity plays a special role so we separate the cases where infinity is an asymptotic value from those where it is not.  For these functions, the dynamics are unchanged by conjugation by an affine transformation.  Thus  suitably normalized solutions have a natural embedding into $\C^{p-1}$

  In his proof of the above theorem, (see also \cite{DK}, Section 1), Nevanlinna shows that for $f \in \mathcal N_p$, a neighborhood of  infinity is divided into exactly $p$ disjoint sectors, $W_0, \ldots, W_{p-1}$, each with angle $2\pi/p$.  Each of these is an asymptotic tract for one asymptotic value.  The sectors are separated by the rays that bound them, and these are Julia rays.  Although two tracts may map to the same asymptotic value, tracts in adjacent sectors must map to different asymptotic values.

 Another immediate corollary of Nevanlinna's theorem is that the family $\mathcal{N}_p$ is topologically closed.  That is,
\begin{cor}\label{cor:Nev}  If $f$ is topologically conjugate to a meromorphic function $g$ in $\mathcal{N}_p$, and if $f$ is meromorphic, then it is also in $\mathcal{N}_p$.
\end{cor}

 \section{The family $\calm_{p,q,r}$}\label{Mpqr}

 A   family that is more general than $\mathcal{N}_p$ is the family of functions
 \[  \tilde\calm_{p,q,r} = \{ f = P \circ g \circ Q \, | \, g \in \mathcal{N}_r, \, P,Q \mbox{   polynomials of degrees } p,q  \}.  \]

 This family is also topologically closed.  We have\footnote{This result is proved in  \cite{FK}.  Since the proof is short and  the ideas used are relevant we include it here.}
 \begin{propA}  If $h$ is topologically conjugate to a meromorphic function $f=  P \circ g \circ Q$ in $\tilde\calm_{p,q,r}$, and if $h$ is meromorphic, (with essential singularity at infinity),  then it is also in $\tilde\calm_{p,q,r}$;  that is there is a  function $\tilde{g} \in  \mathcal{N}_r $ and  polynomials $\tilde P, \tilde Q$ of degrees $p,q$ respectively such that $h = \tilde{P} \circ \tilde{g} \circ \tilde{Q}$.
 \end{propA}

  The proof uses quasiconformal mappings and the Measurable Riemann mapping theorem.   We refer the reader  to the standard literature (\cite{A,LV} and for a more dynamical discussion, \cite{BF}.
 \begin{proof}
We prove the theorem for $h$ quasiconformally conjugate to $f$.   It then follows from Theorem 3.3 of \cite{KK1} that it is true for $h$ topologically conjugate to $f$.

 Let $ \phi^{\mu}$ be a quasiconformal homeomorphism with Beltrami coefficient $\mu$ such that
\[
h=  \phi^{\mu} \circ f \circ (\phi^{\mu})^{-1} =\phi^{\mu}\circ P \circ g \circ Q                 \circ (\phi^{\mu})^{-1}
\]
 is meromorphic.   We can use $P$ to pull back the complex structure defined by $\mu= \bar{\partial}\phi^\mu / \partial \phi^\mu$ to obtain a complex structure $\nu=P^*\mu$ such that the map
\[ \tilde P= \phi^{\mu} \circ  P \circ (\phi^{\nu})^{-1} \]
is  holomorphic.   Note that this is not a conjugacy since it involves two different homeomorphisms.
Similarly,  we can use $g$ to pull back the complex structure defined by $\nu$ to obtain a complex structure $\eta=g^*\nu$ such that the map
\[ \tilde g= \phi^{\nu} \circ  g \circ (\phi^{\eta})^{-1} \]
is meromorphic.  Again this is not a conjugacy.
The function
\[
\tilde{P} \circ \tilde{g}= \phi^{\mu} \circ P \circ (\phi^{\nu})^{-1} \circ \phi^{\nu} \circ g \circ (\phi^{\eta})^{-1}
\]
is a composition of meromorphic functions and is meromorphic.    Now set
\[
\tilde Q=\phi^{\eta} \circ Q_0 \circ (\phi^{\mu})^{-1}.
\]
Although $\tilde Q$ is  not conjugate to $Q$, we were given that
\[h=  \phi^{\mu} \circ P \circ (\phi^{\nu})^{-1} \circ \phi^{\nu} \circ g \circ (\phi^{\eta})^{-1}\circ \phi^{\eta} \circ Q_0 \circ (\phi^{\mu})^{-1} = \tilde{P} \circ \tilde{g} \circ \tilde{Q}\]
  is meromorphic  so that $ \tilde{Q}$ is also meromorphic.

The main point here is that although $\tilde g$ is not   a conjugate of $g$, since the quasiconformal maps $\phi^{\mu}$ and  $\phi^{\nu}$   are homeomorphisms,  the map $\tilde g $    is a meromorphic  map with the same topology as $g$; that is, it has $r$ asymptotic values and no critical values so that by Corollary~\ref{cor:Nev} $\tilde{g}$ belongs to $\caln_r$.   Similarly,  although $\tilde{Q}$ and $\tilde{P}$ are not  respective  conjugates of $Q$ and $P$,   because the quasiconformal maps $\phi^{\mu}$, $\phi^{\nu}$ and  $\phi^{\eta}$   are homeomorphisms,   the maps $\tilde{Q}$ and $\tilde{P}$ are holomorphic maps with the same topology as $Q$ and $P$; that is, they are respectively a degree $q$ and a degree $p$ branched covering of the Riemann sphere with the same number of critical points and the same branching as   $Q$ and $P$.  The maps $\phi^{\mu}$, $\phi^{\nu}$ and  $\phi^{\eta}$ are defined up to normalization.   If we want to keep  the essential singularity of $h$ at infinity,  we normalize so that $\tilde{Q}$ and $\tilde{P}$ are polynomials of degrees $q$ and $p$ respectively, (that is, infinity is a fixed critical point with respective multiplicities $q-1$ and $p-1$).
  \end{proof}

A corollary to this theorem is
\begin{corA}\label{cor:finitedimspace}  The space of functions $\tilde\calm_{p,q,r}$ has a natural embedding into $\C^{p+q+r+4}$ and the subspace $\calm_{p,q,r}$ has dimension $\C^{p+q+r+3}$.
\end{corA}
\begin{proof} A polynomial of degree $d$ is determined by its $d+1$ coefficients.  If  $g \in \caln_r$, set $S_g=S(g)$.   It is determined by its $r-1$ coefficients.  It is straightforward to check that if $w_1,w_2$ are linearly independent solutions of the equation
\[ w'' + \frac{1}{2} S_g w = 0 \]
then $S(w_2/w_1)=S_g$.    Since the space of solutions to this linear differential equation has dimension $2$,  and the solution $g$ of the Schwarzian equation is the quotient of two such solutions,  $g$ is  determined by $r+2$ parameters.  Requiring that infinity cannot be an asymptotic value restricts one parameter.
Therefore $\tilde\calm_{p,q,r}$ has a natural embedding into $\C^{p+q+r+4}$ and the subspace $\calm_{p,q,r}$ has dimension $\C^{p+q+r+3}$.
\end{proof}

\subsection{Dynamics of meromorphic functions}\label{dynamics}
 Let $f:\mathbb{C}\to \widehat{\mathbb{C}}$ be a transcendental meromorphic function and let $f^n$ denote the $n^{th}$ iterate of $f$, that is $f^n(z)=f(f^{n-1})(z)$ for $n\geq 1$. Then $f^n$ is well-defined, except at  the poles of $f, f^2, \cdots, f^{n-1}$, which form a countable set.  These points have finite orbits that end at infinity.

 The basic objects studied are the \emph{Fatou set} and \emph{Julia set} of the function $f$. The Fatou set $F(f)$ of $f$ is defined by
\[  F(f) = \{ z\in \mathbb{C} \ | \ f^n \text{ is defined and normal in a neighborhood of } z\}, \] and the Julia set by
\[ J(f)=\widehat{\mathbb{C}}\setminus F(f).  \]

Note that the point at infinity is always in the Julia set.  If $f$ is a meromorphic function with more than one pole, then the set of prepoles,  $\mathcal{P}=\cup_{n\geq 1}f^{-n}(\infty)$ is infinite. By the Picard theorem, $f^n$ is normal on $\widehat{\mathbb{C}}\setminus\overline{\mathcal{P}}$.  Since it is  not normal on $\overline{\mathcal{P}}$, $J(f)=\overline{\mathcal{P}}$, (see also \cite{BKL1}).

A point $z$ is called a periodic point of $p\geq 1$, if $f^p(z)=z$ and $f^k(z)\neq z$ for any $k<p$. The multiplier of the cycle is defined to be $\mu=(f^p)'(z).$ The periodic point is \emph{attracting} if $0<|\mu|<1$, \emph{super-attracting} if $\mu=0$, \emph{parabolic} if $\mu=e^{2\pi i \theta}$, where $\theta$ is rational number, and {\em neutral} if $\theta$ is not rational.  It is \emph{repelling} if $|\mu|>1$.

If  $D$ is a component of the Fatou set, then $f(D)$ is either a component of the Fatou set or  a component missing one point. For the orbit of $D$ under $f$, there are only two cases:
\begin{itemize}
\item there exist integers $m\neq n\geq 0$ such that $f^m(D)\subset f^n(D)$, and $D$ is called \emph{eventually periodic};
 \item for all $m\neq n$, $f^n(D)\cap f^m(D)=\emptyset$, and $D$ is called a \emph{wandering domain}.
\end{itemize}

Suppose that $\{D_0, \cdots, D_{p-1}\}$ is a periodic cycle of Fatou components, then either:
\begin{enumerate} [\rm (a)]
\item The cycle is (super)attracting: each $D_i$ contains a point of a periodic cycle with multiplier $|\mu|<1$ and  all points in each domain $D_i$ are  attracted to this cycle.
If $\mu=0$, the critical point itself belongs to the periodic cycle and the domain is called super-attracting.
\item The cycle is parabolic: the boundary of each $D_i$ contains a point of a periodic cycle with multiplier $\mu=e^{2\pi i q/m}$, $(q,m)=1$, $m$ a divisor of $p$,  and all points in each domain $D_i$ are
attracted to this cycle.
\item  The components of the cycle are  Siegel disks:  that is, each $D_i$ contains a point of a periodic cycle with multiplier $\mu=e^{2\pi i \theta}$,  where $\theta$ is irrational and there is a holomorphic homeomorphism
mapping each $D_i$ to the unit disk $\Delta$, and conjugating the first return map
$f^p$ on $D_i$ to an irrational rotation of $\Delta$. The preimages under this conjugacy
of the circles $|\xi|=r, r<1$, foliate the disks $D_i$ with $f^p$ forward invariant
leaves on which $f^p$ is injective.
\item  The components of the cycle are Herman rings: each $D_i$ is holomorphically homeomorphic to a standard
annulus and the first return map is conjugate to an irrational rotation of
the annulus by a holomorphic homeomorphism. The preimages under this conjugacy of the circles $|\xi|=r, 1<r<R$, foliate the disks with $f^p$ forward invariant leaves on which $f^p$ is injective.
\item  $D_i$ is an essentially parabolic (Baker) domain: the boundary of each $D_i$
contains a point $z_i$ (possibly $\infty$) and  $f^{np}(z)\to z_i$ for all $z\in D_i$, but $f^p$ is not
holomorphic at $z_i$. If $p=1$, then $z_0=\infty$.
\end{enumerate}

\begin{defn} Define the {\em post-singular set}  of $f$ as the closure of the orbits of the singular values; that is,
\[ P\cals_f=\overline{\cup_{n=0}^\infty f^n(\cals_f)}. \]
For notational simplicity, if a pre-pole $s$ is a singular value, $\cup_{n=0}^\infty f^n(s)$ is a finite set and includes infinity.
\end{defn}

Each non-repelling  periodic cycle  is associated to some singular point.  In particular we have, see e.g. \cite{M2}, chap 8-11 or \cite{Ber}, Sect.4.3,

\begin{thm} If $\{D_0, \cdots, D_{p-1}\}$ is an attracting, superatttracting, parabolic or Baker  periodic cycle of Fatou components, then for some $i =0, 1, \ldots, p-1$, $D_i$ contains a singular value.   If $\{D_0, \cdots, D_{p-1}\}$  is a cycle of rotation domains (Siegel disks or Herman rings) the boundary of each $D_i$ contains the accumulation set of some singular value.
\end{thm}

We use the following  definition of a hyperbolic function in $\tilde\calm_{p,q,r}$\footnote{This definition is equivalent to standard definitions of hyperbolic functions because for these function $\cals_f$ is finite.}
\begin{defn}\label{def:hyperbolic}  An  $f \in \tilde\calm_{p,q,r}$
 is called \emph{hyperbolic} if
\[P\cals_f\cap J(f)=\emptyset.\]
\end{defn}

Note that if all the singular values of $f$ are attracted to attracting or super-attracting cycles then $f$ is hyperbolic.

 For a discussion  of hyperbolicity for more general meromorphic functions,  see the discussion in \cite{Z} and the references therein.

Singularly finite maps may have Baker domains but

\begin{thm}\label{NoWD}\cite{BKL4}
If $\cals_f$ is finite, then there are no wandering domains in the Fatou set.
\end{thm}

\subsection{Holomorphic families}

In this paper, we are interested in the family $\calm_{p,q,r} \subset \tilde\calm_{p,q,r}$    of meromorphic functions in $\tilde\calm_{p,q,r}$ for which infinity is not an asymptotic value.  For each choice of  triples $(p,q,r)$ this is  an example of a holomorphic family.   Below we state the general definitions and results we need.

 \begin{defn}[Holomorphic family]
A {\em holomorphic family} of meromorphic maps over a complex manifold $X$  is a  map
$\calf:X \times \C \rightarrow \hat\C$, such that $\calf(x,z)=:f_x(z)$ is meromorphic for all $x\in X$ and $x \mapsto f_x(z)$ is holomorphic for all $z\in \C$.
\end{defn}

\begin{defn}[Holomorphic motion]
A {\em holomorphic motion} of a set $V \subset \hat\C$ over a connected complex manifold with basepoint $(X,x_0) $ is a map  $\phi: X  \times V \rightarrow \hat\C$ given by
 $(x,v) \mapsto \phi_x(v) $ such that
 \begin{enumerate}[\rm (a)]
\item  for each $v \in V$ , $\phi_x(v)$ is holomorphic in $x$,
 \item
  for each  $x \in X$,  $\phi_x(v) $ is an injective function of  $v \in V$, and,
  \item  at $x_0$, $\phi_ {x_0} = {\rm Id}$.
  \end{enumerate}

A holomorphic motion  of a set $V$ respects the dynamics of the holomorphic  family $\calf$ if $\phi_x(f_{x_0} (v)) = f_x(\phi_x(v))$  whenever both $v$ and $f_{x_0}(v)$ belong to $V$ .

\end{defn}

The following equivalencies are proved for rational maps in \cite{McM} and extended to the transcendental setting in \cite{KK1}.

\begin{thm} \label{jstable}
Let $\calf$ be a holomorphic family of meromorphic maps with finitely many singularities, over a complex manifold $X$, with base point $x_0$. Then the following are equivalent.
\begin{enumerate} [{\rm (a)}]
\item The number of attracting cycles of $f_x$ is locally constant in a neighborhood of  $x_0$.
\item There is a holomorphic motion of the Julia set  of $f_{x_0}$ over a neighborhood of $x_0$ which respects the dynamics of $\calf$.
\item If in addition, for $i=1, \ldots,N$, $s_i(x)$ is are holomorphic maps parameterizing the singular values of $f_x$, then the  functions $x \mapsto f_x^n(s_i(x))$ form a normal family on a neighborhood of $x_0$.
\end{enumerate}
\end{thm}

\begin{defn}[$J-$stability]
A parameter $x_0\in X$ is a $J$-stable parameter for the family $\calf$ if it satisfies any of the above conditions.
\end{defn}

The set of non $J$-stable parameters is precisely the set where bifurcations occur, and it is often called the {\em bifurcation locus} of the family $\calf$, and denoted by $\calb_X$. In families of maps with more than one singular value, however, it makes sense to consider subsets of the bifurcation locus where only some of the singular values are bifurcating, in the sense that the  families  $\{g_n^i(x):= f_x^n(s_i(x))\}$ are normal in a neighborhood of $x_0$ for some values of $i$, but not for all.  We define
\[
 \calb_X(s_i)= \{x_0\in X \mid  \{g_n^i(x)\} \text{\ is not normal in any neighborhood of $x_0$}.\}
\]

In this paper we investigate {\em dynamically natural}  one dimensional slices of the holomorphic family $\calm_{p,q,r}$.      In these slices,  roughly speaking, all the dynamic phenomena but one are fixed and the last is determined by a ``free asymptotic value''.   We will study the components of the complement of the bifurcation locus in these slices.  The parameters in them are $J$-stable.

Precisely, (see \cite{FK})
 \begin{defn}\label{dynnatslice}
 A one dimensional subset $\Lambda \subset X $ is a {\em dynamically natural slice} with respect to $\calf$ if the following conditions are satisfied.

\begin{enumerate}[\rm (a)]
\item $\Lambda$ is holomorphically isomorphic to the complex plane punctured at points where the function is not in the family;  for example, points where the number of singular values is reduced.  The removed points are called {\em parameter singularities}.  By abuse of notation we denote the image of $\Lambda$ in $\C$ by $\Lambda$ again,  and denote the variable in $\Lambda$ by $\lambda$.

\item The singular values are given by distinct holomorphic functions $s_i(\lambda)$, $i=1,\ldots, N-1$, and an asymptotic value $v_{\la}$ that is an affine function of $\la$.\footnote{This is only for convenience of exposition and can be arranged by a holomorphic change of coordinates in $X$.}  We call $v_\la$ the {\em free asymptotic value}; we require that $\calb_\Lambda(v_\la) \neq \emptyset$.
\item The poles (if any) are given by distinct holomorphic functions $p_i(\lambda)$, $\lambda \in \Lambda$, $i\in\Z$.
\item The critical values and some of the asymptotic values are  attracted to an attracting or parabolic cycle\footnote{If the attracting cycle is parabolic, the functions in the slice are not hyperbolic but the results hold.  See \cite{FK} for further discussion}  whose period and multiplier are constant for all $\la\in \Lambda$.   We call these {\em dynamically fixed singular values}.  The remaining singular values include $v_\la$;  if there is only one remaining value, it may have arbitrary behavior.  If there are several,  they  satisfy a relation that persists throughout $\Lambda$ so that  the remaining dynamical behavior is controlled by the behavior of $v_\la$.  Examples of how this may work are discussed in Section~\ref{tanpq}. %
\item  Suppose  $v_{\la_0}$ is attracted to $\cala_{\la_0}$,  the basin of attraction of an attracting  cycle that does not attract any of the dynamically fixed singular values.   Then
the slice $\Lambda$ contains, up to affine conjugacy,  all meromorphic maps $g:\C\to \chat$ that are quasiconformally conjugate to  $f_\lo$ in $\C$ and conformally conjugate to $f_\lo$ on  $\C \setminus \cala_{\la_0}$.
\item $\Lambda$ is maximal in the sense that if $\Lambda'= \Lambda \cup \{\la_0\}$ where $\la_0$ is a parameter singularity, then $\Lambda'$ does not satisfy at least one of the conditions above.
\end{enumerate}
\end{defn}
In the components of the complement of the bifurcation locus in these slices  $v_\la$ is attracted to an attracting cycle $\mathbf a$ of fixed period; this is {\em the period of the component}.  We distinguish two cases:
\begin{enumerate}[\rm(i)]
\item   $\bfa$ does not attract one of the dynamically fixed singular values.  We call these {\em Shell components} and denote   the individual components by $\Omega$ and the collection by $\cals$.
\item $\bfa$ attracts one of the dynamically fixed singular values in addition to attracting  $v_\la$.  We call these {\em Capture components}.  Their properties are very different from those of the shell components and we leave a discussion of them to future work.
\end{enumerate}

\section{Shell components}\label{shell}
  The properties of shell components are described in detail in \cite{FK}.   We summarize them here.  We assume $\calf_\la$ is the restriction of a holomorphic family in $\calm_{p,q,r}$ to a   dynamically natural slice $\Lambda$ and denote a function in $\calf_\la$ by $f_\la$.  We need the following definitions

\begin{defn}  Suppose that for some $\la$, $\{a_1, a_2, \ldots, a_{k-1},  a_{0}  \}$ satisfies $f_\la^i(a_i)=a_{i+1}  \mod k$, $i=1, \ldots, k-1$ where $a_1=v_\la$ and $a_{0}=\infty$;  that is, $v_\la$ is a prepole of order $k-1$.  Then $\la$ is called a {\em virtual cycle parameter of order $k$}.   This is justified by that fact that if $\gamma(t)$ is an asymptotic path for $v_\la=a_1$ and if $h$ is the branch of the inverse of $f^k$ taking $\infty$ to the prepole $a_1$ then
\[ \lim_{t \to \infty} f^k(h(\gamma(t))=v_\la=a_1. \]
We can think of the points as forming a {\em virtual cycle}.
\end{defn}

\begin{defn} Let $\Omega$ be a shell component in $\Lambda$ and let
\[\bfa_\la=\{a_0, a_1, \ldots, a_{k-2},  a_{k-1} \} \]
be the attracting cycle of period $k$ that attracts $v_\la$.   Suppose that as $n \to \infty$,  $\la_n \to \las \in \partial{\Omega}$ and     the multiplier $\mu_{\la_n}=\mu(\bfa_{\la_n})=\Pi_0^{k-1} f'(a_i^n) \to 0.$   Then $\las$ is called a {\em virtual center} of $\Omega$.
\end{defn}
Since the attracting basin of the cycle $\bfa_\la$ must contain $v_\la$, we will assume throughout that the points in the cycle are labeled so that $v_\la$ and $a_1$ are in the same component of the immediate basin.

The next theorem collects the main the results in \cite{FK} about shell components in a dynamically natural slice $\Lambda$ for a holomorphic family of transcendental functions with finite singular set, none of whose asymptotic values is at infinity.  Note that Theorem A of the introduction is part (c) of the theorem.

\begin{thm}\label{shell components}
Let $\Omega$ be a shell component in $\Lambda$.  Then if $\D^* = \{ z \, : \, 0<|z| <1 \}$
\begin{enumerate}[(\rm a)]
\item The map $\mu_\la: \Omega \rightarrow \D^*$ is a universal covering map.  It extends continuously to $\partial\Omega$ and $\partial\Omega$ is piecewise analytic;  either $\Omega$ is  simply connected and   $\mu_\la$ is infinite to one or $\Omega$ is isomorphic a punctured disk and the puncture is a parameter singularity.

\item There is a unique virtual center on $\partial\Omega$.   If the period of the component is $1$, the virtual center is at infinity.
\item  Every (finite) virtual center of a shell component is a virtual cycle parameter and any virtual cycle parameter on the boundary of a shell component is a virtual center.
\end{enumerate}
\end{thm}

Here we prove a stronger theorem for slices of the family $\calm_{p,q,r}$; recall that all their asymptotic values  are finite.  Note that because the functions are of the form  $f(z) =Q \circ g \circ P(z)$, if $v$ is an asymptotic value of $g$, then   $Q(v)$ is an asymptotic value of $f$.    There are  $p$ distinct asymptotic tracts corresponding to each asymptotic tract of $v$  so that there are $pr$ distinct asymptotic tracts at infinity separated by the Julia directions.        The asymptotic values, tracts and Julia directions depend holomorphically on the parameters.
At each pole of $f$ of order $k$ there are $kqr$ pre-asymptotic tracts and the pull-backs of the rays in the Julia directions separate them.

\begin{thmB}
\label{nearby attracting cycle}
Let $\Lambda$ be a dynamically natural slice for the meromorphic family $\calm_{p,q,r}$ consisting of meromorphic  functions of the form  $f_\la=Q \circ g \circ P$ all of whose asymptotic values are finite,  and let $\las$ be a virtual center parameter of order $k$.  Then $\las$ is on the boundary of a shell component of order $k$ in $\Lambda$.  That is,  in any neighborhood of $\las$ there exists $\la \in \Lambda$ such that $f_\la$ has an attracting cycle of period $k$.
 \end{thmB}

As an immediate corollary to Theorem~\ref{shell components} and Theorem B is
\begin{corB}
 In a dynamically natural slice of  $\calm_{p,q,r}$,  every virtual cycle parameter is a virtual center and vice versa.
 \end{corB}

\begin{remark} The essence of the theorem is that the dynamic picture at the poles is reflected in the parameter picture at the virtual center parameters.\footnote{ This is analogous to the situation for Misiurewicz points in the parameter plane for quadratic polynomials. }   If infinity is an asymptotic value, both the dynamics and parameter pictures are different.
\end{remark}

\begin{proof}   Since $\las$ is a virtual center parameter of order $k$,  $f_{\las}$ has a virtual cycle  $\ba^*=\{a_1^*=v_{\las},\ldots, a_{k-1}^*, a_{0}^*=\infty \}$. For each  $j=2, \ldots, k-1$, the cycle uniquely determines
  a branch of the inverse of $f_{\las}$, $f_{\las,j}^{-1}$  such that $f_{\las,j}^{-1}(a_{j}^*)=a_{j-1}^*$.  By abuse of notation, for readability, we drop the $j$ and denote all of these branches by $f_{\las}^{-1}$.   Because we are in a dynamically natural slice of a holomorphic family,
  the analytic continuations of the $f_{\las}^{-1}$, denoted by $f_{\la}^{-1}$ are well defined.

   Note however, that in a neighborhood  of the asymptotic value $a_1^*$, the inverse branch of $f_{\las}$ is not  uniquely determined.
Since $v_{\las}$ is part of a virtual cycle, a neighborhood $U$ of $v_{\las}$ has at least one pre-image that is in an asymptotic tract. If the multiplicity of $v_{\las}$ is one,  then by definition there is a unique asymptotic tract that is determined by the virtual cycle of $f_{\las}$; we denote it by $A_{\las}$ and take as $f_{\las}^{-1}$ the branch that maps $U$ to $A_{\las}$. Then taking the analytic continuation of this $f_{\las}^{-1}$ we obtain $A_{\la}$ as the analytic continuation of $A_{\las}$. If the multiplicity  of $v_{\las}$ is greater than one there will be a choice among the tracts corresponding to $v_{\las}$ (and hence the inverse branch $f_{\las}^{-1}$). Similarly, if there is more than one asymptotic value that varies with $\lambda$ (and satisfies a functional relation with $\lambda$), we choose the tract (or one of them if there is more than one) corresponding to the free asymptotic value\footnote{See the examples in Part 2.}. In the argument below, we take $A_{\las}$, or choose one of the tracts as $A_{\las}$, along with the corresponding branch $f_{\las}$. Since, in general, the asymptotic value is not omitted, there may also be infinitely many inverse branches of the neighborhood that are bounded. We don't need to concern ourselves with them here.

  Because we are in a  dynamically natural slice we have the following holomorphic functions:
\begin{enumerate}
\item   $ v(\la)=v_\la$ is the free asymptotic value of $f_\la$.
\item If $p^*=a_{k-1}^*$, then $p(\la)=p_\la$ is the holomorphic function defining the pole of $f_\la$ such that   $p(\las)=p^*$.
\item  Note that for $\la$ in a   neighborhood $V  \subset \Lambda$  of $\las$, the affine function $\la \mapsto v_{\la}$   determines a corresponding neighborhood $U$ of $v_\las$ in the dynamic plane of $f_{\las}$ and vice versa.    Define the map
$h: V \rightarrow \C$ by $h(\la)=f^{k-2}_\la(v_\la)=h_\la$.  Then if $V$ is small enough,
both $p_\la$ and $h_\la$ are in   $\hat U \subset U$,  a small neighborhood of $p^*$   in the dynamic plane of $f_{\las}$.
\item  In the dynamic plane of $f_{\la}$, set $u(\la)= f_\la^{-(k-2)}(p_\la)=u_\la$.  Then  $u_\la$ is the preimage of the pole $p_\la$ in a neighborhood of the asymptotic value $v_\la$.
\item Each $u_\la$ has infinitely many inverses in the asymptotic  tract $A_\la$; denote these by $w_{\la, j}$, $j \in \Z$.
\end{enumerate}
   The main ideas for the proof are first to use the relation between $\lambda$ and the free asymptotic value $v_{\la}$ to carefully choose and fix a $\lambda$ close to $\las$ and then to construct a domain  $\mathbb T \subset A_\la$ such that $f_\la^k(\mathbb T) \subset {\mathbb T}$.    It will then follow from Schwarz's lemma that $f_\la^k$ has an attracting fixed point in $\mathbb T$.

 \begin{figure}[htb!]
\begin{center}
\includegraphics[height=3in]{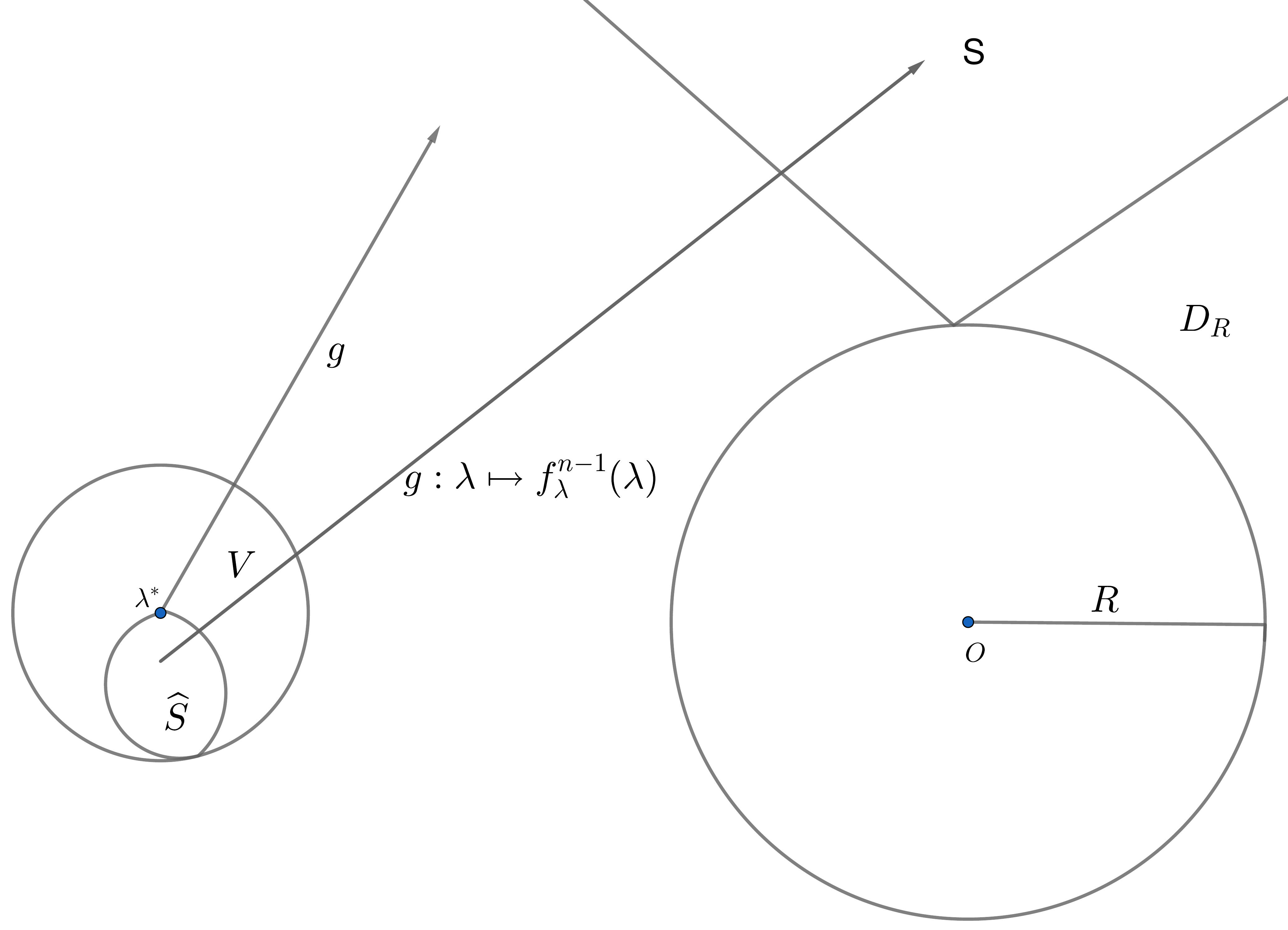}

\end{center}
\caption{\small The map $g_{\la}$ on parameter space.  $S$ is a sector inside all the asymptotic tracts $A_{\la}$, $\la \in V$.  Note that $g(\la^*)= \infty$.}
\label{param}
\end{figure}

{\bf Choosing $\lambda$:}  (See figure~\ref{param}.) The asymptotic tract $A_{\las}$ lies between two Julia rays $r_1^*, r_2^*$ and these span an angle $\theta_{pr}=2\pi/pr$.  We choose a small enough neighborhood $V$ of $\las$ such that  for each $f_\la, \la \in V$,  the Julia rays $r_1(\la),r_2(\la)$ of  $f_\la$,   lie within $\delta=\delta(V)$ of $r_1^*, r_2^*$ respectively.   Each $f_{\la}$ has an asymptotic tract $A_{\la}$ between these rays and we can find a set
$S \subset \cap_{\la \in V} A_\la$   which lies inside the sector in $\hat\C$ bounded by the rays $r_1^*$ and  $r_2^*$ such  that  $S$ is a sector in  $\hat\C$ with vertex at infinity and angle $\theta_{pr}-2\delta$.

 As above, let $U$ be a neighborhood of $v_{\las}$ and let $V$ be the corresponding neighborhood in parameter space.    For each $\la \in V$, $v_\la \in U$   and  $f_\la^{k-1}(v_\la)$  is defined.   Define $g_\la:V \rightarrow \C$ by $g(\la) =f_\la^{k-1}(v_\la)$.   Since $V$ contains $\las$ and $g_\la$ is holomorphic,  we can find $R=R_V$ such that $g(\la)=|f_\la^{k-1}(v_\la)| > R.$    Set $D_R=\{z \in \C \, | \, |z| >R\}$ and let $\hat{S}_V=g_{\la}^{-1}(S \cap D_R )$ be the ``triangular'' subset of $V$ with vertex at $\las$.  Note that the number of such triangular sets is equal to the order of the pole.  If it is one, the set is unique.  If it is greater than one, we choose one  arbitrarily.

 \begin{figure}[htb!]
\begin{center}
\includegraphics[height=3.2in]{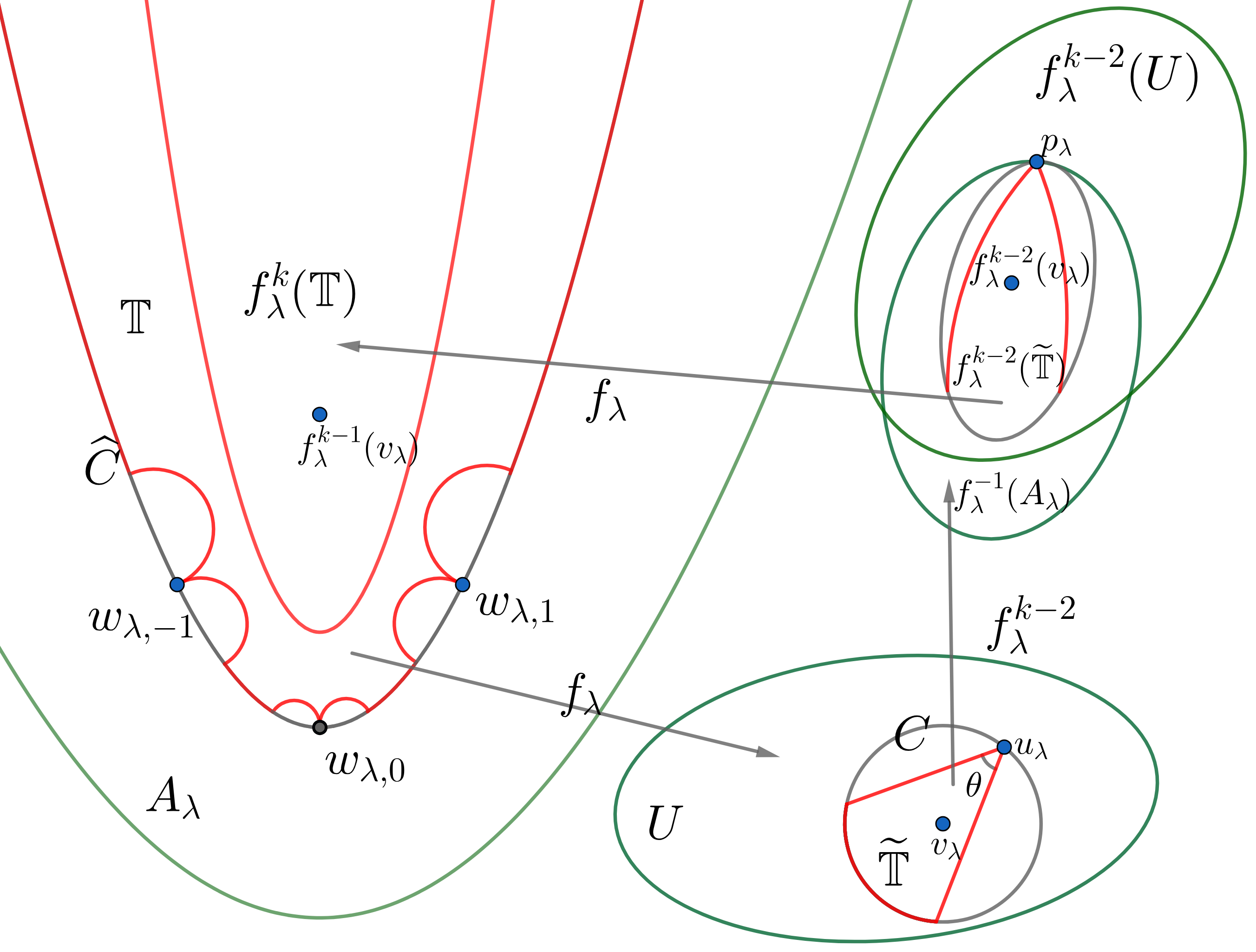}
\end{center}
\caption{\small  The dynamic plane for $f_{\la}$.  The region $f_\lambda^{n}(\mathbb T)$ is contained inside $\mathbb T$.}
\label{dynam}
\end{figure}

{\bf Constructing $\mathbb T$:} (See figure~\ref{dynam}.)  Now we work with a fixed $\la \in \hat{S}_V$ and we set  $U=f_{\la}(A_{\la})$.  Since $S \subset A_{\la}$  we have $f_\la(S) \subseteq U$. In $U$ we have $v_\la$ and $u_\la=f_\la^{-(k-2)}(p_\la)$.   Let $C$ be a circle with center $v_\la$ and radius $|v_\la-u_\la|=\eta_\la$. Taking $\la$  closer to $\las$ if necessary, we may assume $C$ lies in a compact subset of $U$;  that is the disk $\{z \, | \, |z-v_\la|  \leq \eta_\la \} \subset U$.
  Now $\hat{C}=f_\la^{-1}(C) \subset A_\la$ and $f_\la: \hat{C} \rightarrow C$ is an infinite to one cover so $\hat{C}$ contains preimages $w_{\la,j}$, $j \in \Z$,   of $u_\la$.

We want to approximate the distance from a point on $\hat{C}$ to $\partial A_{\la}$.   Let $\phi:A_\la\rightarrow H_l  $
and $\psi: U \setminus\{v_\la\} \rightarrow \D^* $
 be  a conformal homeomorphisms from the left half plane to $A_\la$
 and the punctured disk to $U \setminus\{v_\la \}$ respectively chosen such that $f_\la = \psi^{-1} \circ \exp \circ \phi $.

   Because $\psi$ is a homeomorphism,  $|\psi'(z)|$ is bounded above and below for $z$ in a closed disk  containing  $C$.   Applying Koebe's distortion theorem, there are positive constants $K_1,K_1'$ such that for $z \in C$,
  \[  \frac{ K_1 \eta_\la}{(1+ \eta_\la )^2} <| \psi^{-1}(z)| <  \frac{ K_1' \eta_\la}{(1- \eta_\la )^2}. \]
  We are assuming $\eta_\la$ is small so we may assume it is less than $1/2$.
  Thus we can find closed annuli in $\D^*$ and $U$ containing $\psi(C)$ and $C$ respectively.

  We then have, for appropriate positive constants,   $K_2,K_2'$,
   \begin{equation}\label{hypdens}   K_2 +|\log \eta_{\lambda}|  <| \Re \log \psi^{-1}(z) | <  K_2'+ |\log \eta_{\lambda}|.
   \end{equation}

   This says that the lift $\log \psi^{-1}(C)$ of the circle $C$ lies in a vertical strip $W$ of bounded horizontal width in $H_l.$  Let $\hat{W}$ denote $\phi^{-1}(W) \subset A_{\la}$; it is the lift of the annulus in $U$ to $A_\la$.

   The covering group $\ZZ$ for the exponential map acting on $H_l$ pulls back to an infinite cyclic group $\Gamma$ generated by a map $\gamma: z \mapsto \gamma(z)$ that is the covering group acting on $A_{\la}$ under  the map $f_\la$ so that
   \[  \phi(\gamma(z)) = w + 2 \pi i.   \]

  Thus
\[ \phi'(z) dz = dw \mbox{ and } \phi'(\gamma(z))\gamma'(z)dz=dw. \]

This says that
\[ \frac{dw}{dz}=\phi'(z)= \phi'(\gamma(z))\gamma'(z), \]
and because $|\phi'(z)|$ is bounded in the closure of a  fundamental domain for $\Gamma$ intersected with  $\hat{W}$, it is bounded for all  $w$ in  $\hat{W}$.
Now because $\phi$ is a conformal homeomorphism it preserves the hyperbolic density so that
\[ \rho_{A_{\lambda}}(\gamma(z)) | d\gamma(z)| =\frac{|dw|}{|\Re w|}  \]
\[\rho_{A_{\lambda}}(\gamma(z)) |\gamma'(z) dz|=\frac{|dw +2 \pi i|}{|\Re (w+2 \pi i|)}=\frac{|dw|}{|\Re w|}. \]

Therefore  the hyperbolic density in $A_{\lambda}$ is invariant under the action of $\gamma$.
Equation~(\ref{hypdens}),  says that $|\Re w|$ is comparable to $|\log \eta_{\lambda}|$ which
  gives us   the estimate we want;   that is, there are positive constants $K_3, K_3'$ such that
     \[  \min_{\zeta \in \partial A_\la, y \in \hat C} | y - \zeta |  \sim K_3 +  K_3' | \log \eta_\la |. \]
     Note that because we have a group action on $A_{\la}$, there will be a $\zeta \in \partial A_{\la}$ and a $y \in \hat{C}$ in each fundamental region.
 Now let  $\tilde{\mathbb T}$ be a
 triangular region in $U$ with one vertex at $u_\la$,  two sides  spanning an angle of $\theta< \theta_{pr}/m$, where $m$ is the order of the pole $p_\lambda$, joining $u_{\la}$ to $C$, and third side  an arc of the circle $C$;  the sides are chosen so  that $\tilde{\mathbb T}$ contains $v_\la$ in its interior.  Next set $\mathbb T=f_\la^{-1}(\tilde{\mathbb T})$.   First note that $\mathbb T \subset A_{\la}$ and $f_\la^{-1}(\partial \tilde{\mathbb T})$ is a doubly infinite curve in $A_\la$ that stays a  bounded distance from $\hat{C}$, where the bound depends on $\eta_\la$ and $\theta$.   The inverse images of sides of the triangle form scallops ``above'' $\hat{C}$  (further inside $A_\la$).

 Now look at $ f^{k-2}_\la(\tilde{\mathbb T})$. This is a triangular shaped region contained in $f^{k-2}(U)$ with a boundary point $p_\la$;  it contains $f^{k-2}_\la(v_\la)$. Choose $y \in \widehat{C}$ realizing the minimum above so that $f_\la^{-1}(y)$ is contained in $f^{k-2}(U)$.  Then, as $p_\la$ is a pole of order $m$, we have
 \begin{equation}\label{disttopole}
 |f_\la^{-1}(y) - p_\la| \sim  \frac{K_3}{|\log \eta_{\lambda}|^\frac{1}{m}}.
 \end{equation}
 Since the derivative of $f_{\las}$ along the orbit of $v_{\las}$ from $v_{\las}$ to $p_{\las}$ is bounded and varies holomorphically with $\la$,  the derivative of $f_\la$ along the orbits of $v_\la$ from $v_\la$ to $f^{k-2}_\la(v_\la)$ and $u_\la$ to $p_\la$ are also bounded for $\la \in \hat{S}_V$.   Because $v_\la \in \tilde{\mathbb T}$ and $\eta_\la$ is small, when we map by $f^{k-2}_\la$ the images of $v_\la$ and points on $C \cap \tilde{\mathbb T}$ are comparably close to the pole $p_\la$.   Specifically, for some positive constant $K_4$, we have
 \[ | f^{k-2}_\la(v_\la) - p_\la |  \sim K_4 \eta_\la  <<  \frac{K_3}{ |\log \eta_{\lambda}|^\frac{1}{m}}.\]
Comparing this with the estimate~(\ref{disttopole})
it   follows that
\[ f^{k-1}_\la(C \cap \tilde{\mathbb T} ) \subset \mathbb T  \mbox{  and }
 f^{k-1}_\la(v_\la) \in f_\la^k(\mathbb T).  \]   Therefore  because $v_\la$ is inside $\tilde{\mathbb T}$,
\[  f_\la^k(\mathbb T)  \subset \mathbb T \]
so that by the Schwarz  lemma  $f_\la^k$  has a fixed point in $\mathbb T$.
  \end{proof}

\part{The Extended Tangent Family $\la \tan^p z^q$.}
\label{part2}

In this part of the paper we use the results above to prove that for family
$\calf = \{\la \tan^p z^q \}$, every shell component of period $n>1$ is bounded. A corollary is that every capture component is   bounded as well.  The proof will follow from by studying the period $1$ and period $2$ components.

\section{ The shell components of $\calf$}\
\label{tanpq}
The family $\F= f_\lambda=\lambda \tan^p z^q$ is a subfamily of $\M_{p,q,r}=\{ P\circ g \circ Q \}$ with $P(z)=z^p$ and $Q(z)=z^q$  and $g$ in the one dimensional slice of $\mathcal{N}_2$ consisting of functions that fix $0$ and have symmetric asymptotic values.   The functions in $\F$ have one fixed critical point at $0$ and have either one asymptotic value with mulitplicity $2q$ or
two   asymptotic values with multiplicity $q$  that are   opposite in sign.     Specifically,  the map $\tan z$ has two distinct asymptotic tracts and two asymptotic values, $\pm i$.    Each of these tracts has $q$ pre-images under $Q^{-1}$.
 If $p$ is even,  all $2q$ of the asymptotic tracts map onto punctured neighborhoods of the single point $i^p\lambda=(-1)^{p/2}\lambda$ which is the free asymptotic value $v_\lambda$.
   If $p$ is odd, $q$ of the asymptotic tracts map onto punctured neighborhoods of $i^p \lambda$ and the other $q$ tracts map onto punctured neighborhoods of $(-i)^p \lambda$.  In this case we choose $v_\lambda=i^p\lambda$ as the free asymptotic value.   The other asymptotic value satisfies the relation $v_\lambda' = -v_\la$.

The punctured plane $\lambda \neq 0$ is thus a  {\em dynamically natural slice} in  $\M_{p,q,2}$  and the dynamics are determined by the forward orbit of   $v_{\lambda}$.

  The full set of shell components in this slice is denoted by
  $\mathcal{S}=\{ \Omega  \}  $ and we divide it into subsets
 depending on the period of the cycle as follows:
\begin{defn}
If $pq$ is even
 \[ \mathcal{S}_{n}=\{ \Omega_n |\ f_\lambda\ \text{has an attracting cycle of period  } n\}, \]
  otherwise
 \[ \mathcal{S}_{n}=\{ \Omega_n \ |\ f_\lambda\ \text{has one attracting cycle of period  } 2n \text{ or} \]
\[ \Omega_n' \ |f_\lambda\   \text{ has two attracting cycles of period } n\}.   \]
\end{defn}

For readability we only include the subscript on $\Omega$ if the period is not obvious from the context.

 Figure~\ref{fig:1} shows the parameter plane for the family $f_\lambda(z)=\lambda \tan^2 z^3$.  The period~$1$ shell components are yellow, the period~$2$ shell components are cyan blue.  The capture components are green.
 \begin{figure}[htb!]
\centerline{
\includegraphics[width=0.6\textwidth]{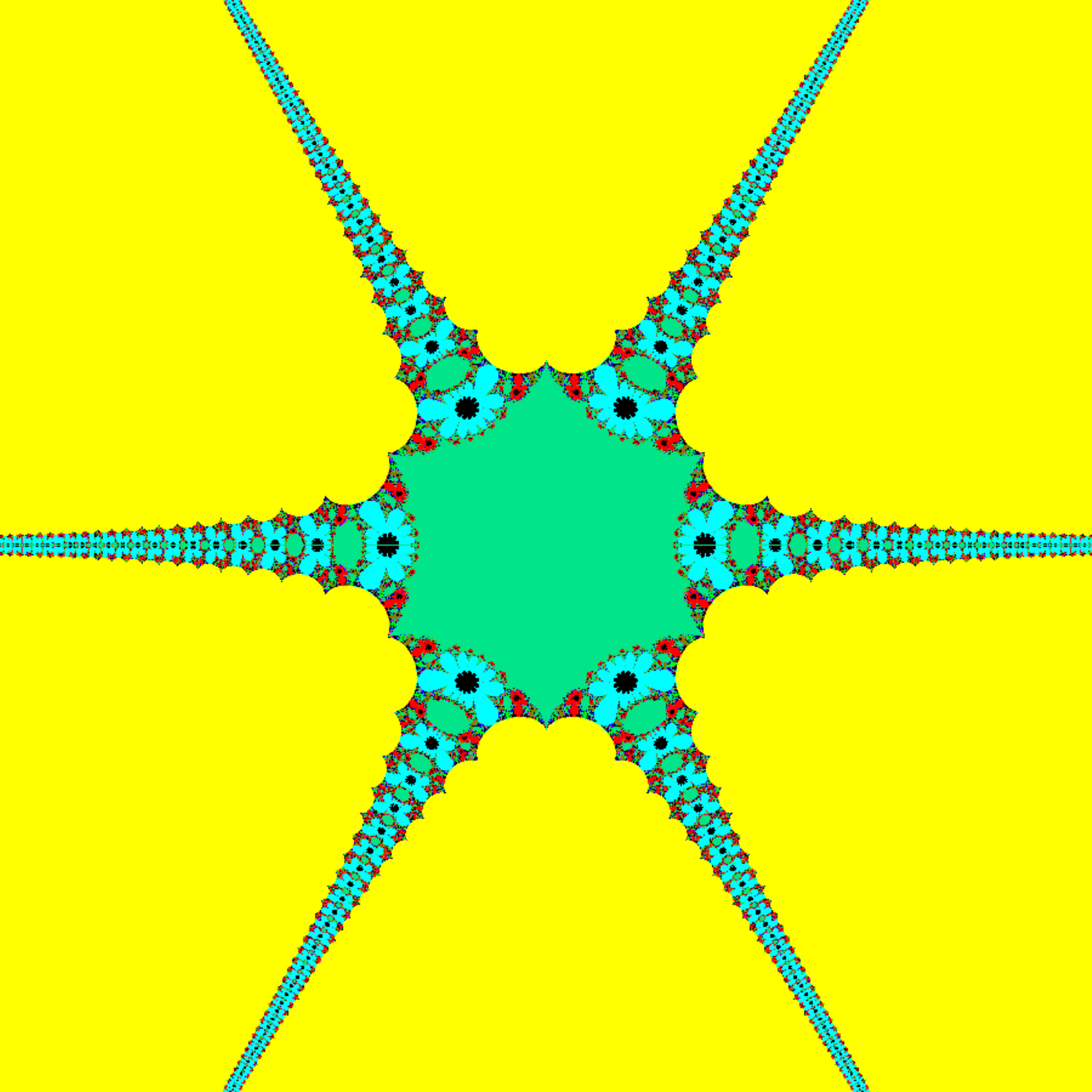}}
\caption{\label{fig:1} \small The parameter plane for $\lambda \tan^2 z^3$.}
\end{figure}

\subsection{Symmetries}
Note  that for any $\lambda$, $f_{\bar{\lambda}}(\bar{z})=\overline{f_{\lambda}(z)}$.    In addition, if $\omega_k, k=0, \ldots, q-1$, are the $q^{th}$ roots of unity,
\[ f_{\omega_k \lambda}(\omega_kz)=\omega_k \lambda \tan^p (\omega_kz)^q=\omega_k f_{\lambda}(z).  \]
 It follows that if  $\Omega \in \cals_n$ then  $\overline{\Omega}, \omega_k \Omega \in \mathcal{S}_n$.

 Suppose $pq$ is even so that  $f_{\lambda}$ has a single attracting cycle $\{z_1, \ldots, z_n\}$ of period $n$ with multiplier $\mu(\lambda)$.  Then  $\{-z_1, \ldots, -z_n\}$ is a cycle for $f_{-\lambda}$ and $\mu(-\lambda)=\mu(\lambda)$.

 If $pq$ is odd then
 \[ f_{\lambda}(-z)=-f_{\lambda}(z)=f_{-\lambda} (z) \mbox{ and } f_{\bar{\lambda}}(z)=\overline{f_{\lambda}(\bar{z})}.  \]
 Assume that $f_{\lambda}$ has two cycles of period $n$. They must be symmetric: that is they are  $\{z_1, \ldots, z_n\}$ and  $\{-z_1, \ldots, -z_n\}$ and they have the same multiplier.

 Now $f_{-\lambda}(z_1)=-z_2$,   $f_{-\lambda}(-z_2)=z_3$, \ldots,  $f_{-\lambda}^m(z_1)=(-1)^mz_{m+1}$, so that  if $n$ is even, $f^n_{-\lambda}(z_1)=z_1$ and $f_{-\lambda}$ also has two cycles of period $n$.  The set of periodic points of $f_{-\lambda}$  is the same as that for $f_{\lambda}$ but they divide into different cycles for $\lambda$ and $-\lambda$; again, $\mu(-\lambda)=\mu(\lambda)$.    If, however, $n$ is odd,
  $f^n_{-\lambda}(z_1)= - z_1$ and $f_{-\lambda}$ has a single cycle of period $2n$; it has multiplier $\mu^2(\lambda)$.
     We summarize this discussion as follows.
  \begin{prop}\label{omegasym}  If $\Omega$ is a hyperbolic component of the $\lambda$ plane, then $-\Omega, \bar{\Omega}$ and $\omega_k\Omega$, $k=0, q-1$ are all hyperbolic components.   That is, the parameter plane is symmetric with respect to reflection in the real and imaginary axes,   rotation by $\pi$ and  rotation by $q^{th}$ roots of unity.

  If $pq$ even,     $\mu(\lambda)=\mu(-\lambda)$ for $\lambda \in \Omega$, while if $pq$ is odd and there are two cycles for  $\lambda \in \Omega$, then each has multiplier $\mu(\lambda)$ and there is a single cycle of double the period for $-\lambda \in -\Omega$ such that for this cycle,  $\mu^2(\lambda)=\mu^2(-\lambda)$.
    \end{prop}

\subsection{Components of $\cals_1$}
In this section we will show that  $\cals_1$ consists of exactly $2q$ unbounded components arranged symmetrically around the origin.
Let $\eta_k=\eta_k^+, k=0, \ldots, q-1$, be the roots of $\eta_k^{q}=i$ and  let $\eta_k^-, k=0, \ldots, q-1$, be the roots of $(\eta^{-}_k)^{q}=-i$.  If $q$ is odd these are labeled so that $ \eta_k^-=-\eta_k^+$ whereas if $q$ is even they are labeled so that  $ \eta_k^-=\overline{\eta_k^+}$.

\begin{thm}\label{comps of cals1} The set $\cals_1$ consists of $2q$ unbounded components, $\Omega_k^{\pm}$, $k=0, \ldots, q-1$, such that each is  symmetric about a ray $\ell_k^{\pm}$  in the direction  $ i^{-p}\eta_k^{\pm}$; that is,  if $\lambda \in \ell_k^{\pm}$, $v_\lambda=   s \eta_k^{\pm}$, $s >0$.

   Moreover, for $\lambda \in \ell_k^{\pm} \cap \Omega^{\pm}$:
\begin{enumerate}[(i)]
\item When both $p, q$ are even:     $f_\lambda$ has a single attracting fixed point $z_\la$  and  $\arg{z_\la} =\arg{v_\lambda} $;
\item   When $p$ is odd and $q$ is even, again $f_\lambda$ has a single attracting fixed point $z_\la$  and  $\arg{z_\la} =\arg{v_\lambda} $ or $\arg(-v_\lambda)$;
\item When  $pq$ is odd: $f_\la$ either has two attracting fixed points, $z_\la$ and $-z_\la$ or a single attracting period two cycle $\{z_\la,-z_\la\}$. Moreover, if $\lambda \in \ell_k^+$,  $f_\lambda$ has two attracting fixed points, $z_\la$ on $\ell_k^+$ and $-z_\la$  on $\ell_k^-$, which attract $v_\lambda$ and $-v_\lambda$ respectively;  if  $\lambda \in \ell_k^-$, then $f_\lambda$ has one attracting cycle of period two, $\{z_\la, -z_\la\}$ which attracts both $v_\lambda$ and $-v_\lambda$ and  we can label the points so that $\arg{z_\la} =\arg{v_\lambda}$.
\end{enumerate}
\end{thm}

\begin{proof}   If $f_\lambda(z) \in \cals_1$ and has an attracting fixed point $z=z_\la$, then using the relation
\begin{equation}\label{fp}
  \lambda=\frac{z}{\tan^p z^q},
  \end{equation}
    we compute that the multiplier $\mu(z)$ is
 \[
 \mu(z)= pqz^{q-1}\tan^{p-1} z^q \sec^2z^q=2 pq\frac{z^q}{\sin 2z^q}.
 \]

  Set $u=2z^q$, so that   $\mu(z)=h(u)=pq u/\sin u$. The locus  $|h(u)|=1$ in the $u=x+iy$ plane consists of two branches, one  in the upper half plane and one  in the lower half plane.  Call the unbounded regions defined by these curves $U^{\pm}$ respectively.  Each is symmetric with respect to both the real and imaginary axes.  Moreover  the regions intersect the imaginary axis in the intervals $u=\pm i r^q$,  $r>r_0$, respectively, where $h( \pm i r_0^q)=1$.  Inside these domains $|h(u)|<1$.    As $|y|\to \infty$, the branches are asymptotic to the curves
  \[
   \frac{\pm e^{|y|}}{2pq}+i y.
   \]
    For each $u$ in $U^{\pm}$ there are $q$ corresponding $z=z_\la$'s and these form $2q$ regions $V_k^{\pm}$, $k=0, \ldots, q-1$ in $\mathbb C$.
     The rays   in the directions $\eta_k^{\pm}$ are axes of symmetry for the $V_k^{\pm}$.  We give the argument for $\eta_k=\eta_k^+$; the argument for $\eta_k^-$ follows similarly.  That is, we want to show that if $ z_1 \in V_k$, then $ z_2=\eta_k^2 \bar{z}_1$ is also in $V_k$. Now if $  z_1 \in V_k$, then $z_1^q \in U^+$ and so $z_2^q=-\bar{z}_1^q$ is also in $U^+$.  Taking appropriate $q^{th}$-roots, the symmetry follows.

     From equation~(\ref{fp}) the images $\Omega_k^{\pm}$ of the $V_k^{\pm}$ are in $\cals$.
      First, the symmetry of each of the  $V_k^{\pm}$'s translates into a symmetry of the corresponding $\Omega_k^{\pm}$. Suppose $z_1, z_2=(\eta_k^+)^2 \overline{z}_1 \in V_k^+$. If $p$ is even we have
    \[
    \la(z_2)=\frac{{\eta_k^+}^2  \bar{z}_1}{\tan^p(- \bar{z}_1^q)} = {\eta_k^+}^2  \overline{\la(z_1)} \]
    which implies that the image of $V_k^+$ is symmetric about the lines in direction $\eta_k^+$.  Similarly for $V_k^-$ about the lines in direction $\eta_k^-$.
    If, however, $p$ is odd,
     \[
        \la(z_2)=\frac{{\eta_k^\pm}^2  \bar{z}_1}{\tan^p( -\bar{z}_1^q)} = -{\eta_k^\pm}^2  \overline{\la(z_1)} \]
    so that the images of $V_k^{\pm}$ are symmetric about the rays  in directions $ i\eta_k^{\pm}$.

 To see that there are $2q$ distinct domains $\Omega_k^{\pm}$, assume that $z$ is on a line of symmetry for $V_k^{\pm}$; that is, $z= s\eta_k^{\pm}$, $s>r_0$.   By equation~(\ref{fp}) we have

 \[\lambda( i\eta_k^{\pm})=\frac{s\eta_k^\pm}{\tan^p(\pm i s^q)}=\frac{s\eta_k^\pm}{(\pm i)^p\tanh^p(s^q)}.\]

     If $pq$ is even,  on the rays where  $z=s\eta_k^+$ and $z=s\eta_k^-$,  the arguments of the corresponding $\lambda$'s under the relation (\ref{fp}) are different, so that the images of $V_k^+$ and $V_k^-$ are different. Therefore there are $2q$ components in $\mathcal{S}_1$  corresponding to $V_k^{\pm}$; these are denoted by $\Omega_k^\pm$.  Moreover,  if $z=s\eta_k^+$ and  $p$ is even,  then $\arg v_\lambda=\arg z$.    If however  $p$ is odd and $q$ is even,  then  if $z=s\eta_k^+$, $\arg v_\lambda=\arg z$ but  if $z=s\eta_k^-$, then $\arg v_\lambda=\arg z +\pi$. In either case,  both $z$ and $v_\la$ are perpendicular to $\lambda$.

   If $pq$ is odd, then  $\lambda(z)$ is an even function,
\[ \frac{z}{\tan^pz^q}=\frac{-z}{\tan^p(-z)^q}, \]
so that both $V_k^{\pm}$ have the same image $\Omega_k^+$.  This gives us $q$ components in $\cals$.  For $\la$ in $\Omega_k^+$,  $f_{\la}$ has two fixed points $z_1$ and  $z_2$ and they both have the same multiplier, $\mu(\la)$.   Moreover on the symmetry lines, $z=  s \eta_k^\pm$, we have $ \arg  v_\lambda=  \arg s \eta_k^\pm$.
Consider the $q$ components $\Omega_k^-=\{\lambda \ | \ -\lambda\in \Omega_k^+\}$.    As we saw in the discussion of Proposition~\ref{omegasym}, when $\lambda\in \Omega_k^-$, $f_\lambda$ has one attracting cycle with period $2$ and multiplier $\mu^2(\la)$ so these components are also in $\cals_1$ and there are  $2q$ components in $\mathcal{S}_1$ in this case as well.
 \end{proof}

 The proof of Theorem~B and the statement of  Theorem~\ref{comps of cals1} tell us something about the structure of the parameter plane for the family $\calf_\la$ in a neighborhood of a virtual center.   We have

 \begin{cor}\label{pqrcomps}  In the parameter plane of the family $\calf_\la$,  there are $2pq$  shell components of period $k$ that meet at every virtual center parameter of order $k>1$.   That is, the virtual center parameter is a common boundary point of these $2pq$ components.  Infinity is a common boundary component of $2q$ components.
\end{cor}

\begin{proof}   In the proof  of Theorem~B we used that fact that the virtual center parameter $\las$ corresponds to a particular prepole $p_{\las}$ of $f_{\las}$ and used one of the asymptotic tracts at $p_{\las}$ in our construction to find the shell component at $\las$.  There are $2pq$ tracts at each prepole; these correspond to the asymptotic values tied to $v_\la$ by the functional relation.  We could have used any one of these in  the argument above
  to obtain a shell component at the virtual center.
  \end{proof}

\subsection{Separating Lines.}\label{sec:seplines}
In this section we discuss the rays in the $\lambda$ plane spanned by the roots of $v_{\lambda}^q=1$ and their negatives.

Let $\omega^\pm_k, k=0, \ldots, q-1$ be the roots of $\omega_k^q=1$ and $(\omega_k^{-})^q=-1$. Since $v_{\lambda}=i^p \lambda$ is the free asymptotic value,  the rays in the $\lambda$ plane we are interested in are those spanned by $i^{-p} \omega^\pm_k$. Denote them by $\tilde{\ell}_k = i^{-p}  r \omega_k^\pm$, $r>0$.

\begin{lemma}\label{seplines}
For the family $\mathcal{F}_\lambda$,
\begin{itemize}
\item[(i)] If $pq$ is even,  none  of the rays in the set
\[\overline{\calr}'= \{  i^{-p} r \omega_k^\pm,  r>0\}_{k=1}^q  \] intersects any component of $\cals$;  that is
$\cals \cap \overline{\calr}'=\emptyset$ for all $k$.
\item[(ii)] if $pq$ is odd, at each virtual cycle parameter $\lambda$ on the ray $\tilde{\ell}_k$, such that $v_{\lambda}$ is a pre-pole of $f_{\lambda}$ of order $n$, $n$ is odd and  there are two components of $\cals_{n+1}$ intersecting the ray.  In one of these components there are two  periodic cycles of  order $n+1$ and in the other there is a single cycle of order $2(n+1)$ attracting both asymptotic values.
\end{itemize}
\end{lemma}

\begin{proof}
Set $\lambda= r i^{-p} \omega_k^\pm$ for some real $r$ so that $v_{\lambda}^q$ is real.   Notice that this implies $f_{\lambda}(v_{\lambda})$ lies in the same   line  as $\lambda$.   We have to look at the parities of $p$ and $q$.

 To prove (i),
suppose first that $p$ is even.  Then $\lambda$ and $v_{\lambda}$ are in the same line.  Moreover,  since $\lambda^q$ is real, it follows that $\lambda$, $v_{\lambda}$, and all its images under $f_{\lambda}$, lie on the same line, although perhaps on opposite sides of the origin.  Therefore this line is invariant under $f_{\lambda}$ and
   the restriction of $f_\lambda$  to this line is conjugate to a real-valued function $f_r(t)=r\tan^pt^q$.

Next suppose $p$ is odd and $q$ is even.  Then $\lambda$ and $v_{\lambda}$ are in perpendicular lines.  Since $q$ is even, $\lambda^q$ is real so
  $(f_{\lambda}(v_{\lambda}))^q$ is real, and
  $f^2_\lambda(v_\lambda)$ is also on the line through $\lambda$, as is the rest of its orbit.  That is, $f_\lambda$ is invariant on the line through $\lambda$ and the orbit of $v_\lambda$ eventually lands on the line.  Again  the restriction of $f_\lambda$  to this line is conjugate to a real-valued function $f_r(t)=r\tan^pt^q$.

We claim that in these cases, $f_{\lambda}$ cannot have an attracting cycle other than the super-attracting fixed point $0$.  We consider the family of real valued functions
 $f_r(t)=r\tan^pt^q$.     On the interval $(0, \sqrt[q]{\pi/2})$, $f_r, f_r'$ and $f_r''$ are all positive and $f_r$ goes from $0$ to infinity.  Therefore there is one fixed point $z_0$ inside this interval.  Because $0$ is a superattracting fixed point, its basin contains the interval $(0, z_0)$.
 By symmetry, about the imaginary axis if $p$ is even and about the origin if $p$ is odd, the basin of $0$  contains the interval $I_0=(-z_0,z_0)$.  Since $f_r''(t)>0$,  $f_r'(t)>1$ for $t$ in $(z_0, \sqrt[q]{\pi/2})$, and $\tan t$ has period $\pi$, $|f'_r(t)|<1$ only in $\mathcal I$, the union of the translates of the interval $I_0$.

 We claim there are no other attracting periodic cycles in $\overline{\calr}'$.   If there were such a cycle, $\{z_1, \ldots, z_n\}$, its points would have to be outside $\mathcal I$ and  its multiplier  would have to satisfy
 $\Pi_1^n |f_r(z_n)|<1$.  This cannot happen since none of these factors can be less than $1$.

 To prove (ii), we have
    both $p$ and $q$ odd.  Then  $\lambda^q $  and $(f_{\lambda}(v_{\lambda}))^q$ are pure imaginary.  Therefore
 \[  f^2_\lambda(v_\lambda) = \lambda \tan^p (f_{\lambda}(v_{\lambda}))^q = \pm i \tanh^p(f_{\lambda}(v_{\lambda})^q \]
 is in the line containing $v_{\lambda}$.   The orbit of $v_{\lambda}$ thus alternates between two perpendicular lines, so if it approaches an attracting cycle, that cycle must have even period and hence $\lambda$ does not belong to $\cals_1$.

Suppose  $\lambda$ is on the ray  spanned by   $r i^{-p} \omega_k$, and is inside one of the components $\Omega_k$ intersecting it.   The multiplier of the cycle is real and monotonic in $\Omega_k$.   The virtual center $\lambda^*$ of $\Omega_k$ is therefore on the ray.  By Corollary~\ref{pqrcomps}, there are $2pq$ components at $\lambda^*$ and since $pq$ is odd, as we move around $\lambda^*$, they alternate between those with two cycles of period $n+1$ and one cycle of period $2(n+1)$ attracting both asymptotic values.  Again since $pq$ is odd, there is one of each type intersecting the ray.
 \end{proof}

\subsection{Components of $\mathcal{S}_2$}\label{per2}
 By Theorem~\ref{shell components},  each component in $\mathcal{S}$ is simply connected and the multiplier map is a universal covering.  Moreover, each component has a virtual center on its boundary where the limit of the multiplier is zero, and if this is finite, it is a virtual center parameter so that the asymptotic value is a pre-pole.   For completeness, we include a proof of the last statement for components of $\mathcal{S}_n$ so that we can use  it to characterize the virtual centers of $\cals_2$.

\begin{lemma}\label{virtcenter} If the  shell component $\Omega\in \mathcal{S}_n$, is bounded, then the virtual center $\lambda^*$ satisfies $f^{n-1}_{\lambda^*}(v_{\lambda^*})=\infty$.  That is,  the asymptotic value is a pre-pole of order $n$ and  the virtual center is a virtual cycle parameter.
\end{lemma}
\begin{proof}
Without loss of generality, we assume $f_\lambda$ has an attracting periodic cycle of period $n$;   the proof is similar if $f_\lambda$ has one attracting cycle of period $2n$. Let $\{z_0, z_1, \cdots, z_{n-1}\}$ be the attracting cycle of $f_{\lambda}$, and suppose $z_0$ lies in an asymptotic tract. The multiplier of the cycle is \begin{equation*}
 \prod_{i=0}^{n-1}f_\lambda'(z_i)=\prod_{i=0}^{n-1}\frac{\lambda pqz_i^{q-1}\tan^{p-1}z_i^q}{\cos^2(z_i^q)}
=(2pq)^{n
}\prod_{i=0}^{n-1}\frac{z^q_i}{\sin (2z_i^{q})}
\end{equation*}
If the multiplier tends to $0$ as $\lambda\to \lambda^*$, then at least one of the  factors $z^q/\sin(2z^q)$ tends to $0$, which in turn implies that the imaginary part $z_0^q$ is unbounded. Since $z_0$ is in the asymptotic tract, it follows that as $\lambda\to \lambda^*$, $|\Im z_0^q|\to \infty$ and $z_1=\lambda \tan^p z_0^q\to v_{\lambda^*}$.   Because $z_0=\lambda \tan^p z_{n-1}^q$ and $\lambda$ is bounded, it follows that $z_{n-1}^q \to k\pi+\pi/2$  as $\lambda \to \lambda^*$
  and  $f^{n-1}_{\lambda^*}(v_{\lambda^*})=\infty$.
\end{proof}

\begin{lemma}\label{centers of cals2}
If $\Omega$ is a component of $\mathcal{S}_2$, then its virtual center is a solution of $v_\lambda^q=k\pi+\pi/2$ for some integer $k$.
\end{lemma}

\begin{proof}

By Lemma~\ref{virtcenter}, we only need to show that each component of $\mathcal{S}_2$ is bounded.

Suppose there is an unbounded component $\Omega\in \mathcal{S}_2$ and $f_{\lambda}$ has a  periodic cycle  $\{z_1,z_0\}$.  We will show that if  $\Omega$ is unbounded then   both points  lie in asymptotic tracts, and therefore that $p$ and $q$ are both odd.   This in turn  implies that  $f_\lambda$ has only one attracting cycle of period $2$ so that  $\Omega\in \mathcal{S}_1$ and not in  $\Omega\in \mathcal{S}_2$.

Let $\lambda(t) \to \lambda^*$, $t \to \infty$ be a path in $\Omega$ such that the multiplier of the periodic cycle $\mu(\lambda(t)) \to 0$.  If $\Omega$ is unbounded,   the multiplier of the periodic cycle has absolute value $1$ for any finite point on the boundary.  Therefore,  ``virtual center'' $\lambda^*$ of $\Omega$ is at infinity. From the proof of  Lemma \ref{virtcenter},  the point $z_0(\lambda(t))$ of the cycle  must be in the asymptotic tract and $|\Im z_0(\lambda(t))^q|\to \infty$ and thus $z_1(\lambda(t))=\lambda \tan^pz_0(\lambda(t))^q \asymp i^p \lambda$ is also goes to infinity.

Suppose $z_1(\lambda(t))^q_1=x_1(t)+iy_1(t)\asymp i^{pq}\lambda(t)^q$.  We claim that if $y_1(t)$  goes to infinity, then $z_1(\lambda(t))$ is also in an asymptotic tract.  If it is not, then $x_1(t)\to \infty $ but $y_1(t)$ stays bounded. Take a sequence  $z_1(t_k)$ such that $x_1(t_k)=k\pi$. Then
\[
\tan (x_1(t_k)+iy_1(t_k))=\tan(iy_1(t_k))=i\tanh (y_1(t_k))=B_ki,
\]  where $B_k$ is real and bounded.

Since $z_0(t_k) =\lambda \tan^p (z_1(t_k))^q=\lambda (B_ki)^p$,   $z_0(t_k)^q/z_1(t_k)^q=B_k^{pq}\in\mathbb{R}$;  that is, $z_0(t_k)$ is asymptotically parallel to $z_1(t_k)$ with a bounded ratio so that the imaginary part of $z_0^q(t)$ remains bounded.
This contradicts our assumption that $z^q_0(t)$ is unbounded, so it follows that $\Omega$ is bounded as claimed.
\end{proof}
\section{Theorem~C}\label{thmC}

We now have all the ingredients to prove Theorem~C.

\begin{thmC} All the components of $\cals_n$, $n>1$ are bounded.
\end{thmC}

\begin{proof}
We proved in Theorem~\ref{comps of cals1}  and Lemma~\ref{seplines} that for the
  family $\mathcal{F}_\lambda$,   $\mathcal{S}_0$, consists of   $2q$ unbounded components, $\{ \pm \Omega_1, \cdots, \pm \Omega_{q} \}$.  These are symmetric about the rays $\overline{\R}=\{ \pm r \eta_k,  r> r_0\}_{k=1}^q$ and separated by the rays $\overline\R'=\{\pm r \omega_k, r> 0 \}_{k=1}^{q}.$

  The boundary of each $\Omega_k$ is an analytic curve defined by the relation $|f_{\lambda}(z_0)|=1$ where $z_0$ is the fixed point or, if $pq$ is odd and there is a single period $2$ cycle, $\{z_0,z_1\}$, the relation $|f_{\lambda}(z_0)f_{\lambda}'(z_1)|=1$.   We assume now that there are two fixed points.   The discussion for the period $2$ cycle is essentially the same.   Along  $\partial \Omega_k$ there is a sequence of points $\nu_{k,m}$ where $f_{\lambda}(z_0)=-1$ and the fixed points have become parabolic.   There is a standard bifurcation at each $\nu_{k,m}$ where a new cycle of period two appears and so there is a bud component $\Omega_{2,k,m}$ of $\cals_2$ tangent to $\Omega_k$ there.  By Lemma~\ref{centers of cals2},  $\Omega_{2,k,m}$ is bounded and has a virtual center, say $s_{m,k}$ on the ray in $\overline{\R}$ separating $\Omega_{k}$ from the next one in order around the origin; for argument's sake assume it is  $-\Omega_{k+1}$.

If we now look at the points on the boundary of $-\Omega_{k+1}$, there is a sequence $\nu_{k+1,m'}$ where the multiplier of the cycle is $-1$ and there is a bud component $\Omega_{2,k+1,m'}$ with virtual center $s_{m',k+1}$ on the same ray of  $\overline{\R}$  separating $\Omega_k$ and $-\Omega_{k+1}$.   Choose $\nu_{k+1,m'}$ so that $s_{m',k+1}=s_{m,k}$.    We can do this since the boundaries of $\Omega_k$ and $-\Omega_{k+1}$ are both asymptotic to the ray containing the centers.

At $s_{m,k}$ there are $2pq$ shell components and $\Omega_{2,k,m}$ and $W_{2,k+1,m'}$ are two of them.  Now we draw a curve in $\Omega_{2,k,m}$ from $\nu_{k,m}$ to $s_{m,k}$ and another from $s_{m,k}$
to $\nu_{k+1,m'}$ in $\Omega_{2,k+1,m'}$.   We then choose another parabolic point $\nu_{k+1,m''}$ on the boundary of $-\Omega_{k+1}$ whose bud component $\Omega_{2,k+1,m''}$ has a center $s_{m'',k+1}$ on the next ray in order around the origin.  We draw a simple curve in $-\Omega_{k+1}$ from $\nu_{k+1,m'}$ to $\nu_{k+1,m''}$.  Continuing around,  in the next component $\Omega_{k+2}$,  we find the bud component that shares $s_{m'',k+1}$ as its center and draw a simple curve through the bud components joining $-\Omega_{k+1}$ and $\Omega_{k+2}$.   We continue in this manner until we get back around to the original $\Omega_k$.   We join all the curves in the $\pm \Omega_k$ and the bud components between them.   The result, $\gamma_m$ is a simple closed around the origin.

In this way, choosing the $\nu_{k,m}$ carefully and systematically, we can create a nested sequence of curves $\gamma_m$ around the origin.   Any other component of $\cals$ is disjoint from the components $\pm \Omega_k$ and the bud components $\Omega_{2,k,m}$ and so must lie inside one of the $\gamma_m$ and is therefore bounded.

\end{proof}
An immediate corollary of the above proof is
\begin{cor}
All the capture components in the dynamically natural slice are bounded.
\end{cor}

\vspace*{20pt}
\noindent Tao Chen, Department of Mathematics, Engineering and Computer Science,
Laguardia Community College, CUNY,
31-10 Thomson Ave. Long Island City, NY 11101.\\
\noindent Email: tchen@lagcc.cuny.edu

\vspace*{5pt}
\noindent Linda Keen, Department of Mathematics, and  CUNY Graduate
School, New York, NY 10016, and Department of Mathematics, Lehman College of CUNY,
Bronx, NY 10468 \\
\noindent Email: LINDA.KEEN@lehman.cuny.edu, linda.keenbrezin@gmail.com.

\end{document}